\documentclass[12pt]{amsart}

\usepackage[english]{babel}
\usepackage[utf8]{inputenc}
\usepackage{amsmath, amssymb, amsthm, amsfonts}
\usepackage{ulem}
\usepackage{tikz-cd}
\usepackage{hyperref}

\newtheorem{theorem}{Theorem}[section]
\newtheorem{lemma}[theorem]{Lemma}
\newtheorem{proposition}[theorem]{Proposition}

\theoremstyle{definition}
\newtheorem{definition}[theorem]{Definition}

\theoremstyle{remark}
\newtheorem{remark}[theorem]{Remark}

\newcommand{\Hyp}{\operatorname{Hyp}}
\newcommand{\Ell}{\operatorname{Ell}}

\newcommand{\PSL}{\operatorname{PSL}}
\newcommand{\Aut}{\operatorname{Aut}}
\newcommand{\Hom}{\operatorname{Hom}}
\newcommand{\Epi}{\operatorname{Epi}}
\newcommand{\Red}{\operatorname{Red}}
\newcommand{\Fix}{\operatorname{Fix}}

\newcommand{\bbQ}{\mathbb{Q}}

\newcommand{\bbF}{\mathbb{F}}

\hypersetup{pdftitle={Redundant generation in PSL2(Qp)}, pdfauthor={Yair Glasner, Amit Levinson-Sela}}
\title[Redundant generation in $\PSL_2\left(\bbQ_p\right)$]{Redundant generation in $\PSL_2\left(\bbQ_p\right)$}

\author[Glasner]{Yair Glasner}
\address{Ben Gurion University of the Negev.
	Department of Mathematics.
	Be'er Sheva, 8410501, Israel.
}
\email{yairgl@bgu.ac.il}

\author[Levinson-Sela]{Amit Levinson-Sela}
\address{Ben Gurion University of the Negev.
	Department of Mathematics.
	Be'er Sheva, 8410501, Israel.
}
\email{levinsam@post.bgu.ac.il}

\date{}

\begin{document}

\begin{abstract}
Given a free group on $n$ generators, $F_n$, and a topological group $G$, $\Aut\left(F_n\right)$ acts on $\Epi\left(F_n,G\right)$, the set of homomorphisms $f:F_n \rightarrow G$ with dense image. 
$f\in \Epi\left(F_n,G\right)$ is called redundant if there is a proper free factor $A < F_n$ whose image under $f$ is dense in $G$.
For $n \ge 3$, $G = \PSL_2 \left(\bbQ_p\right)$, we prove that every $f\in \Epi\left(F_n,G\right)$ with torsion-free image is redundant, and give sufficient conditions for $f$ to be redundant even when its image is not torsion-free. 
\end{abstract}
\maketitle

\section{Introduction}
Let $V = \operatorname{span}\left(v_1,\dots,v_n\right)$ be a finite dimensional vector space. Then $\dim V \le n$, and Gaussian elimination gives a way to compute $\dim V$ precisely: performing it on the matrix with rows $\left(v_1,\ldots,v_n\right)$, one arrives at a new matrix whose nonzero rows form a basis of the given space. It is natural to attempt similar methods in the setting of (topologically) finitely generated groups, a procedure which is loosely referred to as the ``Nielsen method''.

Let $F_n=\left< x_1,x_2,\dots,x_n \right>$ be a free group on $n$ generators; its group of automorphisms, $\Aut\left(F_n\right)$, admits a generating set, known as Nielsen transformations. These are the following automorphisms, reminiscent of elementary row operations:
\begin{itemize}
    \item $x_i \rightarrow x_i ^{-1}$ for some $1\le i \le n$, $x_k \rightarrow x_k$ for all $k\neq i$;
    \item $x_i \rightarrow x_j,x_j \rightarrow x_i$ for some $1\le i \neq j \le n$, $x_k \rightarrow x_k$ for all $k \neq i, k \neq j$;
    \item $x_i \rightarrow x_ix_j$ or $x_i \rightarrow x_jx_i$ for some $1\le i\le n$, $j\neq i$, $x_k \rightarrow x_k$ for all $k\neq i$.
\end{itemize}

Let $G$ be a topological group. We denote by $\Hom(F_n,G)$ the set of all homomorphisms from $F_n$ to $G$. We will refer to $\Hom(F_n,G)$ as the representation variety of $F_n$ into $G$\footnote{When $G$ is an algebraic group this name is often reserved for the scheme parameterizing all such representations, but we will not use this terminology here.}. The natural identification $\Hom(F_n,G) \cong G^n$  endows the representation variety with the product topology. Denote by $\Epi(F_n,G) \subset \Hom(F_n,G)$ the set of all topological epimorphisms, i.e. those homomorphisms whose image is dense. Under the above identification, $\Epi(F_n,G)$ is identified with tuples $(g_1,\ldots,g_n) \in G^n$ that generate a dense subgroup of $G$.

There is a continuous action of $\Aut\left(F_n\right)$ on $\Hom\left(F_n,G\right)$ by precomposition: if $\varphi\in\Aut\left(F_n\right),f\in\Hom\left(F_n,G\right)$, then $\varphi\left(f\right)=f\circ\varphi^{-1}\in\Hom\left(F_n,G\right)$. On $G^n$ this action can be understood explicitly as the action of Nielsen transformations on the set of $n$-tuples. Clearly this action restricts to an action on the invariant set $\Epi\left(F_n,G\right)$. The study of this action has two basic motivations: to understand $\Aut\left(F_n\right)$, and to understand generating sets of a given group $G$.

Recall that a subgroup $A<F_n$ is called a free factor if there exists $B<F_n$ such that $F_n=A*B$ (it is a proper free factor if $B\neq 1$). Equivalently $A<F_n$ is a free factor if and only if there exists $\varphi \in \Aut\left(F_n\right), 1\le k \le n$ such that $A =  \left< \varphi \left(x_1\right) , \dots , \varphi \left(x_{k}\right) \right>$. An element $x\in F_n$ is called primitive if $\left<x\right>$ is a free factor.
\begin{definition} With the above notation, $f\in \Epi\left(F_n,G\right)$ is called a \textit{redundant} representation if there exists a proper free factor $A<F_n$ such that $\overline{f(A)}=\overline{f(F_n)} = G$. Analogously, a topological generating tuple $S=(g_1, \dots , g_n)$ is called \textit{redundant} if there exists a Nielsen-equivalent tuple $S'=(g_1',\dots,g_n')$ such that $G = \overline{\left< g_2',\ldots, g_n' \right>}$. We denote by $\Red\left(F_n,G\right) \subset \Epi\left(F_n,G\right)$ the subset of redundant representations.
\end{definition}

 Note that if $(g_1,\dots,g_n )$ is redundant, then $n > d\left(G\right)$. It is natural to ask whether the converse also holds: when $n > d\left(G\right)$, is every $n$-tuple which generates a dense subgroup necessarily redundant? For $G=\PSL_2 \left(\bbQ_p\right)$ we give a positive answer under restrictions on the finite order elements of the dense subgroup.
\begin{theorem} \label{thm:torsion_free_redundancy}
    Let $n\ge3$, $p$ be a prime and $G = \PSL_2(\bbQ_p)$. Then every $f \in \Epi(F_n,G)$ for which $\Gamma := f(F_n)$ is torsion-free is redundant.
\end{theorem}

In the case that $\Gamma$ does have torsion, 
we can prove a similar statement under some geometric restrictions:

\begin{theorem} \label{thm:finite_fix_redundancy}
    Let $n \ge 5$, $p \ge 5$ be a prime and $G = \PSL_2(\bbQ_p)$. Assume, for $f \in \Epi(F_n,G)$, that $\Fix(f(w))$ consists of finitely many vertices whenever $f(w) \in G$ is nontrivial of finite order. Then $f$ is redundant.
\end{theorem} 

$\Aut\left(F_n\right)$-representation varieties have been studied in different settings, with much of the investigation revolving around a conjecture of Wiegold and Goldman to the effect that the action of $\Aut(F_n)$ should be transitive for $n > d\left(G\right)$; one should consult the survey papers~\cite{Lu1} (in particular for the case of finite groups) and~\cite{Ge1} (in particular for the case of linear algebraic groups). 
It appears that the notion of redundancy is central to understanding the dynamics of these representation varieties, as discussed in~\cite{Ge1}. Furthermore, redundancy has come to be studied in its own right in different settings; for connected Lie groups and complex algebraic groups it is investigated in~\cite{CoVi}.
For $G=\PSL_2 \left(K \right)$ where $K$ is a nonarchimedean local field, and for $\Aut^{0}\left(T\right)$ the group of automorphisms of a regular infinite tree without edge inversions, the first author (\cite{Gl}) showed that almost every generating set of size $n\ge3$ is redundant, and thus the action of $\Aut\left(F_n\right)$ on $\Epi\left(F_n,G\right)$ is ergodic. 
In contrast, our Theorems~\ref{thm:torsion_free_redundancy} and~\ref{thm:finite_fix_redundancy} restrict to $\PSL_2\left(\bbQ_p\right)$ and impose further restrictions on the dense subgroup, but give a stronger answer in those cases --- one without reference to measure. 

In Section~\ref{sec:subgroups} we recall important facts about the action of $G = \PSL_2(\bbQ_p)$ on its Bruhat--Tits tree. Section~\ref{sec:weidmann} is dedicated to Weidmann's application of the Nielsen method for groups acting on trees. Finally in Sections~\ref{sec:proofs} and~\ref{sec:torsion} we give the proofs of the main theorems. 

This paper is essentially the M.Sc. thesis of the second author under the joint supervision of Prof. Michael Brandenbursky and the first author. We are greatly thankful to Michael for his constant support for this project. Both authors were partially funded by ISF grant 3187/24 as well as by BSF grant 2024317.

The results in this paper were obtained without use of AI; in fact, they all appear in the M.Sc. dissertation of the second author, which was completed in 2025. However, we did use LLM aid in reviewing the final draft of the paper version. 

\section{Subgroups of \texorpdfstring{$\PSL_2\left(\bbQ_p\right)$}{PSL2(Qp)}} \label{sec:subgroups}
We recall useful facts about subgroups of $\PSL_2\left(\bbQ_p\right)$, as proven in~\cite{CoSc}. 
Throughout this paper we consider subgroups of $\PSL_2\left(\bbQ_p\right)$ for some prime $p$, and denote by $T=T_{p+1}$ the associated Bruhat--Tits tree, with $\partial T$ its boundary. Occasionally we consider $\PSL_2\left(K\right)$ over a general nonarchimedean local field $K$.
We denote by $\Ell$, $\Hyp$ the sets of elliptic and hyperbolic elements of $G$, respectively. For $g\in \Ell$ let $\Fix(g)$ denote the subtree of fixed points. 

Actions on trees can be classified into four types by their invariant sets in the tree and its boundary, as follows:

\begin{definition} \label{def:action_classification} Let $\Gamma < \Aut^{0}\left(T\right)$ be finitely generated. Then exactly one of the following holds:
    \begin{enumerate}
        \item Every $\gamma \in \Gamma$ is elliptic, and $\Gamma$ fixes a point $v\in T$, in which case we say that $\Gamma$ is elliptic;
        \item There is some $\gamma \in \Gamma$ which is hyperbolic, and a pair of points $\left\{x,y\right\} \subseteq \partial T$ invariant under the action of $\Gamma$ (possibly but not necessarily fixed pointwise), in which case we say that $\Gamma$ is lineal;
        \item There is some $\gamma \in \Gamma$ which is hyperbolic, and a unique point $x \in \partial T$ which is fixed by $\Gamma$, in which case we say that $\Gamma$ is focal;
        \item There is no point in $T$ or in $\partial T$ which is fixed by $\Gamma$, and no pair of points in $\partial T$ which is invariant under $\Gamma$, in which case we say that $\Gamma$ is irreducible.
    \end{enumerate}
\end{definition}
\begin{remark} \label{rem:limit_set}
    In the irreducible case, the limit set $L = \overline{\{\gamma_+ \ | \ \gamma \in \Gamma \cap \Hyp\}} \subset \partial T$ is a perfect compact $\Gamma$-invariant subset of $\partial T$ (where $\gamma_+$ denotes the attracting point of $\gamma$). The set of pairs $\{(\gamma_+,\gamma_-) \ | \ \gamma \in \Gamma \cap \Hyp\}$ is dense in $L \times L$. In fact, given two hyperbolic elements $\eta,\theta \in \Gamma$, for a large enough value of $n$ the element $\gamma_n = \eta^n g \theta^n$ will be hyperbolic with $(\gamma_n)_- \rightarrow \theta_-$ and $(\gamma_n)_+ \rightarrow \eta_+$ as $n \rightarrow \infty$, whenever $g \in \Gamma$ is chosen such that $g \theta_+ \ne \eta_-$.
\end{remark}

In the case of subgroups of $\PSL_2\left(K\right)$ acting on its Bruhat--Tits tree, we can relate these geometric properties of the action with the intrinsic structure of the group:

\begin{lemma}[{\cite[Lemma 6.6]{CoSc}}] \label{lem:nonelementary_zariski_dense} 
    Let $\Gamma < G = \PSL_2\left(K\right)$ be finitely generated. Then $\Gamma$ is irreducible $\iff$ $\Gamma$ is Zariski-dense and non-precompact.
\end{lemma}

\begin{proposition}[{\cite[Proposition 3.4 and Corollary 3.5]{CoSc}}] \label{prop:finite_order_fixed_sets}
    Let $a\in \PSL_2\left(\bbQ_p\right)$ be a nontrivial element of finite order $n$, and suppose that $p \ge 5$. Then:
    \begin{itemize}
        \item $\Fix\left(a\right)$ either consists of a single vertex, or is a bi-infinite geodesic;
        \item $\Fix\left(a^k\right) = \Fix\left(a\right)$ for every $1 \le k < n$.
    \end{itemize}
\end{proposition}

\begin{lemma}[{\cite[Lemma 6.7]{CoSc}}] \label{lem:zariski_trichotomy}
    Let $\Gamma<\PSL_2\left(\bbQ_p\right)$ be Zariski-dense, nondiscrete and non-precompact. Then $\Gamma$ is dense.
\end{lemma}

\begin{remark} \label{rem:subfields}
Of course, many of the results in~\cite{CoSc} are stated for $\PSL_2(K)$, where $K$ is a general local field. However, we rely heavily on the specific stated versions of the last two lemmas which hold only for $K = \bbQ_p$. 
For example, in the last lemma, if $L<K$ is a local subfield, one has to also allow for examples like $\PSL_2(L) < \PSL_2(K)$, which can be disregarded if $K = \bbQ_p$. We refer to~\cite[Proposition 8.4]{Sh} for a discussion of the more general case. 
\end{remark}

We recall also the well-known lemmas:
\begin{lemma}[Tits' Lemma, {\cite[I.6 Proposition 26]{Se}}] \label{lem:tits_lemma}
    Let $a,b\in \Aut^0\left(T\right)$ be elliptic elements such that $\Fix\left(a\right) \cap \Fix\left(b\right) = \emptyset$. Then $ab$ is hyperbolic.
\end{lemma}

\begin{proposition}[Helly's Theorem] \label{prop:helly}
    Let $T_1,\dots,T_n$ be subtrees of $T$ such that $T_i \cap T_j \neq \emptyset$ for all $1\le i,j\le n$. Then $\bigcap_{1\le i \le n}T_i \neq \emptyset$.
\end{proposition}

\section{Around Weidmann's theorem}
\label{sec:weidmann}
Many of the early, non-topological, proofs for the Nielsen--Schreier theorem are surprisingly delicate. Nielsen's original proof, which relies on his namesake method (\cite{Ni1},~\cite{Ni2}), works (only for) a finitely generated subgroup $\left< x_1, \ldots, x_n \right>$. Performing Nielsen moves on the given set of generators, as in the general strategy outlined in the introduction, he shows that it is always possible to reach a ``minimal'' generating set of the form $\left< y_1, \ldots, y_d,1,\ldots,1 \right>$, consisting of a free generating set and trivial elements. This method was later interpreted geometrically, using the action of the free group on its Cayley tree, and the geometric approach was generalized in many papers (see~\cite{Zi},~\cite{CoZi},~\cite{PeRe},~\cite{We} and the references there). 

We will rely on the latter paper, in which  Weidmann treats more general groups acting, without inversions, on a tree $T$. Thus consider a group $G = \left< \mathcal M \right> < \Aut^{0}(T)$ generated by a partitioned finite set
\[
\mathcal M=(S_1,\ldots,S_\ell;H), \qquad {\text{setting }} U_i=\left< S_i\right>.\]
Weidmann allows $\mathcal M$ to be modified only by using special Nielsen moves, which respect the partition structure. Namely:
\begin{itemize}
\item Either conjugate
an entire block $S_i$ by an element of the subgroup generated by the complementary
entries,
\item or replace an entry $h\in H$ by $a h b$, where
$a,b$ are in the subgroup generated by the remaining entries. 
\end{itemize}
An important role in Weidmann's theory is played by the trees $T_i$. $T_i$ is defined as the minimal $U_i$-invariant subtree containing all the fixed points of all nontrivial elements of $U_i$. We will work under the very special assumption that these trees degenerate to one single point, and in particular that the groups $U_i$ are elliptic. Here is the special case of Weidmann's theorem that we will use
\begin{theorem}[{\cite[Theorem 7]{We}}] \label{thm:weid}
Let $\mathcal M=(S_1,\ldots,S_\ell;H)$ be a partitioned set of generators for $G < \Aut^{0}(T)$. Assume that each $U_i = \left< S_i \right>$ is an elliptic subgroup such that 
\[T_i := \bigcup_{1 \ne u \in U_i} \Fix(u) = \{x_i\}, \qquad {\text{is a single vertex.}}\] 
Then, after
finitely many of the preceding moves, one obtains an equivalent
partitioned set 
\(\mathcal M'=(S'_1,\ldots,S'_\ell;H')\), with corresponding data $U'_i = \left< S'_i \right>$, $T'_i = \{x'_i\}$, for which one of the following holds:
\begin{enumerate}
    \item \label{w:free_prod}
    \(G = U'_1*\cdots *U'_\ell*F(H')\) and 
    moreover, \(G_{x'_i}=U'_i \ \forall i\),
    \item \label{w:new_elliptic} Some entry belonging to $H'$ is elliptic,
    \item \label{w:combine:elliptics} $x'_i=x'_j$ for some $i\neq j$.
\end{enumerate}
\end{theorem}

\begin{proof}
By Weidmann's minimization lemma (\cite[Lemma 9]{We}), the partitioned set is equivalent to
one satisfying his minimality condition. Since the associated tree of
$U_i$ is $T_i=\{x_i\}$, the three obstructions in Weidmann's theorem
simplify as follows.

The first obstruction is precisely $x_i=x_j$ for some $i\neq j$.
The second says that $hx_i=x_i$ for some $h\in H$, and hence that $h$
is elliptic. The third also says directly that an element of $H$ is
elliptic. If none of these occurs, Weidmann's theorem gives the stated
free product decomposition. The stabilizer assertion is the stronger
conclusion obtained in the proof of that theorem.
\end{proof}

\noindent \textit{An important remark is in order here.} Weidmann's theorem is usually used for locally infinite trees. In fact, the theorem degenerates in the locally finite case whenever $U_i$ is an infinite elliptic group. Indeed for any vertex $x_i$ fixed by the group and $R>0$, there is a finite index subgroup of $U_i$ fixing the ball $B_{T}(x_i,R)$ so that $T_i = T$, and the whole Theorem~\cite[Theorem 7]{We} trivializes. Locally finite trees become interesting here only if all the groups $U_i$ are finite, which is the case of interest to us. 

An additional assumption that is very special to our setting is that the trees $T_i$ reduce to a single point. This is exactly the reason behind our restrictive assumptions that 
$K = \bbQ_p$, $p \ge 5$ and that our group does not contain finite order elements with an infinite fixed point set. Under these assumptions it follows directly from Proposition~\ref{prop:finite_order_fixed_sets} that $T_i = \{x_i\}$ reduces to one single point whenever $U_i = \left< S_i \right>$ is a nontrivial finite group.

\section{Proofs}
\label{sec:proofs}

From Lemmas~\ref{lem:zariski_trichotomy} and~\ref{lem:nonelementary_zariski_dense} it follows that in order for a subgroup to be dense it suffices for it to be nondiscrete and irreducible. Thus we will show that the nondiscreteness and irreducibility of a dense subgroup $\Gamma < \PSL_2\left(\bbQ_p\right)$ can be successively detected in some free factor, each of them separately.  We begin by showing how any free factor can be upgraded, at the cost of some Nielsen transformations and one extra generator, to a free factor whose image acts irreducibly. 

\begin{proposition} \label{prop:irreducible_detection}
    Let $\mathcal{M} = \left(S_1, \dots ,S_\ell\right)$ be a finite partitioned set of generators for an irreducible subgroup $\Gamma < \Aut^0\left(T\right)$, $\ell \ge 3$. Then, after finitely many moves of the form $S_i \rightarrow gS_i$ for $1 \le i \le \ell$, $g\in S_j$, $j \neq i$, one obtains an equivalent partitioned set $\mathcal{M}' = \left(S'_1,\dots, S'_\ell\right)$ such that $S'_\ell = S_\ell$ and $\left<S'_i,S'_j\right>$ is irreducible for some $i\neq j$.
\end{proposition}

\begin{proof}
    Suppose that $\Gamma$ is irreducible, and that $\left<S_i,S_j\right>$ is elliptic, lineal or focal for every $i,j$. 
    By Proposition~\ref{prop:helly}, there must be some $i < j$ such that $\left<S_i,S_j\right>$ is not elliptic. Hence, by Lemma~\ref{lem:tits_lemma}, there exists $g\in \left<S_j\right>$ such that $\left<gS_i\right>$ is not elliptic. We therefore set $S'_i := gS_i$.
    If $\left<S'_i\right>$ is not irreducible, it must be either lineal or focal.
    \begin{itemize}
        \item If $\left<S'_i\right>$ is focal, let $x\in \partial T$ be its unique fixed point. 
        Since $x$ is not fixed by $\Gamma$, there exists some $j \neq i$ such that $\left<S'_i, S_j \right>$ is not focal; hence it is irreducible.
        \item If $\left<S'_i\right>$ is lineal, let $\left\{x,y\right\} \subseteq \partial T$ be the unique pair of points invariant under it. 
        If there exists $j \neq i$ such that $\left<S_j\right>$ stabilizes neither $x$, nor $y$, nor $\left\{x,y\right\}$, then $\left<S'_i,S_j\right>$ is irreducible. 
        Otherwise, since neither $x$, nor $y$, nor $\left\{x,y\right\}$ is invariant under $\Gamma$, one of the following occurs: either there exist $j_1 < j_2$ such that $\left<S'_i, S_{j_1}\right>$ fixes $x$ and $\left<S'_i,S_{j_2}\right>$ fixes $y$; or there exist $j_1 < j_2$ such that $\left<S'_i,S_{j_1}\right>$ fixes (without loss of generality) $x$ and $\left<S'_i,S_{j_2}\right>$ stabilizes $\left\{x,y\right\}$ but does not fix $x$; or there exist $j_1 < j_2$ such that $\left<S'_i,S_{j_1}\right>$ stabilizes $\left\{x,y\right\}$ but does not fix $x$ and $\left<S'_i,S_{j_2}\right>$ fixes (without loss of generality) $x$.
        In all three of these cases, there exists $g\in \left<S_{j_2}\right>$ such that neither $x$, nor $y$, nor $\left\{x,y\right\}$ is invariant under $\left<gS_{j_1}\right>$. Setting $S'_{j_1} := gS_{j_1}$, we have that $\left<S'_{j_1}, S'_i\right>$ is irreducible.
    \end{itemize}    

    Finally, note that since $i < j$ and $j_1 < j_2$, it must be that $S_\ell$ is unchanged.
\end{proof}

\begin{proposition} \label{prop:upgrade_to_irreducible}
    Let $n\ge 3$ and set $\Delta = \left<g_3, \dots, g_n \right>$. Assume that $\Gamma=\left<g_1,g_2,\Delta\right> < \PSL_2\left(K\right)$ is Zariski-dense and non-precompact and that $\Delta \neq 1$. Then there is a Nielsen-equivalent generating tuple $\left( g'_1,g'_2, g_3, \dots, g_n\right)$ and $\gamma \in \left<g'_1,g'_2\right>$ such that $\left<g'_2, \gamma \Delta \gamma^{-1} \right>$ is Zariski-dense and non-precompact.
\end{proposition}

\begin{proof}
    Setting $S_1 := \left\{g_1\right\}, S_2:=\left\{g_2\right\}, S_3 := \left\{g_3,\dots,g_n\right\}$, Proposition~\ref{prop:irreducible_detection} implies that there exists a Nielsen-equivalent generating tuple $\left(g'_1,g'_2,g_3,\dots,g_n\right)$ such that either $\left<g'_i, \Delta\right>$ is Zariski-dense and non-precompact for some $i\in \left\{1,2\right\}$, or $\left<g'_1,g'_2\right>$ is Zariski-dense and non-precompact. 
    In the former case the proposition holds (after possibly exchanging $g'_1$ and $g'_2$) for $\gamma=1$. 
    In the latter case, by Lemma~\ref{lem:tits_lemma} and Proposition~\ref{prop:helly}, we may assume that $g'_2$ is hyperbolic, and denote its attracting and repelling points by $x,y \in \partial T$. 
    
    For any $\gamma \in \left<g'_1, g'_2\right>$, $\left<g'_2, \gamma \Delta \gamma^{-1} \right>$ is irreducible if and only if neither $x$, nor $y$, nor $\left\{x,y\right\}$ is stabilized by $\gamma \Delta \gamma^{-1}$.
    Equivalently, if and only if neither $\gamma^{-1} x$, nor $\gamma ^{-1} y$, nor $\gamma ^{-1} \left\{x,y\right\}$ is stabilized by $\Delta$. Now, $\Delta$ cannot fix, pointwise, all of the limit set $L$ of $\left< g_1',g_2'\right>$, because an element fixing three points is trivial, while $L$ is uncountable. 
    Let $z \in L$ and $\delta \in \Delta$ be such that $\delta$ does not fix $z$, and let $O \subseteq L$ be an open neighborhood of $z$ such that $\delta O \cap O = \emptyset$.
    Then by Remark~\ref{rem:limit_set} we can find some hyperbolic element $\theta \in \left<g'_1,g'_2\right>$ whose repelling point is in $O$ and whose attracting point is neither $x$ nor $y$. 
    Taking $\gamma = \theta^k$ for a large enough value of $k$ yields $\gamma^{-1}x,\gamma^{-1}y \in O$, so that $\delta \in \Delta$ takes both of these points away from themselves and from each other.
    Therefore, $\left<g'_2, \gamma \Delta \gamma^{-1} \right>$ is Zariski-dense and non-precompact.
\end{proof}

\begin{proof}[Proof of Theorem~\ref{thm:torsion_free_redundancy}]
Set $g_i := f(x_i)$. 
By assumption, $\Gamma = \left<g_1,\dots,g_n\right>$ is dense and torsion-free. 
Theorem~\ref{thm:weid} yields a Nielsen-equivalent generating tuple $\left(g'_1,\dots , g'_n\right)$ with $g'_n \in \Ell$. 
We may assume $g'_i \ne 1$ for all $i$, as otherwise we are done. 
Proposition~\ref{prop:upgrade_to_irreducible} yields a Nielsen-equivalent tuple $\left(g''_1,g''_2, g_3'' = \gamma g'_3 \gamma^{-1},\dots,g_n''=\gamma g'_n \gamma^{-1} \right)$ such that $\Sigma := \left<g''_2, g''_3, \ldots, g''_n \right>$ is Zariski-dense and non-precompact. 
Since $g''_n = \gamma g'_n \gamma^{-1}$ is elliptic, and since it has infinite order by the assumption that $\Gamma$ is torsion-free, it must be that $\Sigma$ is nondiscrete. 
Thus $\Sigma$ is dense by Lemma~\ref{lem:zariski_trichotomy}, finishing the proof. 
\end{proof}

\section{Results for groups with torsion} \label{sec:torsion}

We now turn to proving Theorem~\ref{thm:finite_fix_redundancy}. We first observe that the answer for the redundancy question for finite groups is known for all finite subgroups of $\PSL_2\left(K\right)$, with a bound on the number of generators:

\begin{lemma} \label{lem:wiegold_PSL}
    Let $\Delta = \left<z_1,\dots,z_n\right> < \PSL_2\left(K\right)$ be finite, $n \ge 3$, and assume the residual characteristic of $K$ does not divide the order of any element in $\Delta$. Then there is a Nielsen-equivalent generating tuple $\left(1, z'_2,\dots,z'_n\right)$.
\end{lemma}
\begin{proof}
    By~\cite[Proposition 3.3]{CoSc}, $\Delta$ is either a cyclic group, a dihedral group, $A_4$, $S_4$, or $A_5$. Cyclic groups, dihedral groups, $A_4$ and $S_4$ are solvable and generated by $\le 2$ elements, so in those cases the lemma holds by~\cite{Du}. $A_5 \cong \PSL_2\left(\bbF_5\right)$, so in that case the lemma holds by~\cite{Gi}.
\end{proof}

\begin{remark} \label{rem:qp_orders}
    If $K=\bbQ_p$ for $p \ge 5$ then by~\cite[p. 972]{Lu2} there is no element of order divisible by the residual characteristic $p$.
\end{remark}

\begin{proposition}
\label{prop:discrete-from-small-free-factors}
Let $p \ge 5$ and let
\[
f:F_n\longrightarrow \PSL_2(\bbQ_p),
\qquad
\Delta=f(F_n).
\]
Assume that:
\begin{enumerate}
    \item the image of every free factor of rank at most $3$ is
    discrete;
    \item every nontrivial finite order element of $\Delta$ fixes a
    unique vertex of the Bruhat--Tits tree.
\end{enumerate}
Then $\Delta$ is discrete.
\end{proposition}

\begin{proof}
Let $\mathcal M = \left(f(x_1),f(x_2),\ldots, f(x_n)\right)$ be the images of the standard generators. If some $f\left(x_i\right)$ is trivial, we throw it away. We will apply Nielsen transformations, without changing notation, in order to simplify the discussion.

Every elliptic entry $g$ of $\mathcal{M}$ is the image under $f$ of a primitive element in $F_n$. Hence $\left< g\right>$ is the image of a rank-one free factor and is discrete. Since it is contained
in a compact vertex stabilizer, it is finite. By assumption, every such $g$ fixes a unique vertex. Thus partition \[\mathcal M = (S_1,S_2,\ldots,S_{\ell};H),\]
where $H$ consists of all hyperbolic generators and the elliptic generators are grouped according to their unique fixed points. 

We claim that every elliptic subgroup $U_i = \left< S_i \right>$ is finite. Assume, without loss of generality, that $S_i = \left(z_1,z_2,\ldots, z_k\right)$, and proceed by induction on $k$. If $k \le 3$ then $U_i$, by assumption, is discrete, and a discrete elliptic subgroup must be finite. Otherwise the group $\left< z_{k-2},z_{k-1},z_{k} \right>$ is finite and, by Lemma~\ref{lem:wiegold_PSL} (which holds due to Remark~\ref{rem:qp_orders}), we can replace these by a Nielsen-equivalent triple $\left(z_{k-2}',z_{k-1}',1\right)$. Thus $U_i$ is actually generated by one fewer generators, and we finish by applying our induction hypothesis. 

We conclude that $T_i = \bigcup_{1\neq g \in U_i} \Fix(g) = \{x_i\}$ reduces to a single vertex for every $1 \le i \le \ell$. 
Thus we may apply Weidmann's Theorem~\ref{thm:weid} and pass to a Nielsen-equivalent tuple satisfying one of the three possibilities mentioned there. In case (\ref{w:free_prod}) the stabilizers $U_i = \Delta_{x_i}$ are finite and $\Delta$ is discrete as required. If we land in cases (\ref{w:new_elliptic}) or (\ref{w:combine:elliptics}), we repeat the same procedure, throwing away any generators that may have become trivial. In case (\ref{w:new_elliptic}) we have reduced the number of hyperbolic generators; in case (\ref{w:combine:elliptics}) we have reduced the number $\ell$ of elliptic groups. So it is easy to see that after repeating the procedure at most $2n$ times, we must arrive at the discrete free product configuration of case (\ref{w:free_prod}).
\end{proof}

\begin{proof}[Proof of Theorem~\ref{thm:finite_fix_redundancy}]
    If every free factor $L < F_n$ with $\operatorname{rank}\left(L\right) \le n-2$ has discrete image, then by Proposition~\ref{prop:discrete-from-small-free-factors} $\Gamma := f\left(F_n\right)$ is discrete, in contradiction to the assumption that it is dense. Hence, setting $g_i := f\left(x_i\right)$, there exists a Nielsen-equivalent generating tuple $\left(g'_1,\dots , g'_n\right)$ such that $\left<g'_3,\dots,g'_{n}\right>$ is nondiscrete. From here we conclude exactly as we did in the proof of Theorem~\ref{thm:torsion_free_redundancy}.
\end{proof}

\bibliographystyle{alpha}
\bibliography{sn-bibliography}
\end{document}